\documentclass[letterpaper,12pt,reqno]{amsart}
\usepackage{graphicx,amsmath,amssymb,amsthm}

\usepackage{srcltx}
\usepackage[latin1]{inputenc}
\usepackage[T1]{fontenc}

\usepackage{fullpage}

\newtheorem{X}{X}[section]

\newtheorem{definition}[X]{Definition}
\newtheorem*{hypothesis}{Hypothesis LI}

\newtheorem{lemma}[X]{Lemma}
\newtheorem{proposition}[X]{Proposition}
\newtheorem{theorem}[X]{Theorem}
\newtheorem{conjecture}[X]{Conjecture}

\theoremstyle{definition}
\newtheorem{remark}[X]{Remark}
\newcommand{\V}{\text{Var}}

\newcommand{\E}{\mathbb E}

\renewcommand{\P}{\text{Prob}}

\title{The distribution of the variance of primes in arithmetic progressions}
\author{Daniel Fiorilli \\ Draft: \today }
\address{Department of Mathematics, University of Michigan, 530 Church Street, Ann Arbor MI 48109 USA}
\email{fiorilli@umich.edu}

\begin{document}

\begin{abstract}

Hooley conjectured that the variance $V(x;q)$ of the distribution of primes up to $x$ in the arithmetic progressions modulo $q$ is asymptotically $x\log q$, in some unspecified range of $q\leq x$. On average over $1\leq q \leq Q$, this conjecture is known unconditionally in the range $x/(\log x)^A \leq Q \leq x$; this last range can be improved to $x^{\frac 12+\epsilon} \leq Q \leq x$ under the Generalized Riemann Hypothesis (GRH).
We argue that Hooley's conjecture should hold down to $(\log\log x)^{1+o(1)} \leq q \leq x$ \emph{for all values of $q$}, and that this range is best possible. We show under GRH and a linear independence hypothesis on the zeros of Dirichlet $L$-functions that for moderate values of $q$, $\phi(q)e^{-y}V(e^y;q)$ has the same distribution as that of a certain random variable of mean asymptotically $\phi(q)\log q$ and of variance asymptotically $2\phi(q)(\log q)^2$. Our estimates on the large deviations of this random variable allow us to predict the range of validity of Hooley's Conjecture.

\end{abstract}
\maketitle



\section{Introduction}


Define the variance of the distribution of primes in arithmetic progressions as\footnote{This variance is sometimes defined with $\psi(x;\chi_0)$ replaced by either $x$ or $\psi(x)$, however these definitions are all equivalent for our purpose, which is to study the validity of the asymptotic $V(x;q)\sim x\log q$.}
$$ V(x;q):= \sum_{\substack{a\bmod q \\ (a,q)=1}} \left| \psi(x;q,a)-\frac {\psi(x,\chi_0)}{\phi(q)} \right|^2.$$
The study of this important quantity has a long history.
One of the major applications of the large sieve is the Barban-Davenport-Halberstam Theorem \cite{B,DH,G}, which asserts that the average of $V(x;q)$ over $1\leq q \leq Q$ is $O(x\log Q)$, in the range $x/(\log x)^A\leq Q\leq x$.
An asymptotic result in this range was first obtained by Montgomery \cite{M}, who showed that for any fixed $A$, we have for $Q\leq x$ that
\begin{equation}
\frac 1Q\sum_{q\leq Q}V(x;q)= x\log x+O_A\left(x\log(2x/Q)+\frac{x^2}{Q(\log x)^A} \right). 
\label{equation montgomery theorem}
\end{equation} 
This estimate was refined by Hooley \cite{Ho1}, who showed that in the same range,
\begin{equation}
\frac 1Q\sum_{q\leq Q}V(x;q)= x\log Q-cx +O_{A}\left(Q^{\frac 14}x^{\frac 34}+\frac{x^2}{Q(\log x)^A} \right), 
\label{equation hooley}
\end{equation} 
where $c=\gamma+\log (2\pi) +1+\sum_p \frac{\log p}{p(p-1)}$. Hooley \cite{Ho1,Ho6} also showed that the error term in \eqref{equation hooley} could be replaced by $O_{\epsilon}(Q^{\frac 14}x^{\frac 34} + x^{\frac 32 +\epsilon}/Q)$ under the Generalized Riemann Hypothesis (GRH), extending the range of validity of the asymptotic $\frac 1Q\sum_{q\leq Q}V(x;q) \sim x\log Q$ to $x^{\frac 12+\epsilon} < Q \leq x$. The most precise result so far under GRH is that of Goldston and Vaughan \cite{GV}, which states that for $Q\leq x$,
\begin{equation}
 \frac 1Q \sum_{q\leq Q} V(x;q) = x \log Q -cx  +O_{\epsilon}\left(Q (x/Q)^{\frac 14+\epsilon} +\frac{x^{\frac 32}}Q (\log 2x)^{\frac 52} (\log \log 3x)^{2} \right). 
 \label{equation Goldston Vaughan}
\end{equation}

As for individual values of $q$, Hooley \cite{Ho2,Ho5} conjectured that in some range of $q\leq x$ we have 
\begin{equation}
V(x;q)\sim x\log q.
\label{equation hooley conjecture}
\end{equation} 
It is still an open problem to determine whether \eqref{equation hooley conjecture} holds in any range of $q$. A lower bound of the conjectured order of magnitude was obtained by Friedlander and Goldston \cite{FG}, who showed that in the range $x/(\log x)^A \leq q \leq x,$
$$ V(x;q)\geq \left(\frac 12-o(1)\right) x\log q.$$
Hooley \cite{Ho4} improved the range of validity of this bound to $\exp(-C\sqrt{\log x}) <q \leq x$.

Similar bounds are known in a wider range of $q$ under GRH \cite{FG,FG2, Ho4}; the most precise lower bound appearing in the literature is 
\begin{equation}
V(x;q) \geq \left( \frac 32 -\frac 1{\alpha} -o(1)\right) x\log q
\label{equation conditional lower bound on V(x;q)}
\end{equation}  
in the range $x^{\frac 23+\epsilon} \leq q \leq x$, where $\alpha=\log q/\log x$. As for upper bounds, Turan \cite{Tu} has shown under GRH that $V(x;q) \ll x(\log x)^4$.
Friedlander and Goldston \cite{FG} have shown that if in addition to GRH one assumes a strong version of the Hardy-Littlewood Conjecture on prime pairs, then \eqref{equation hooley conjecture} holds in the range $x^{\frac 12+\epsilon} \leq q \leq x$.
More precisely, they show that if we assume the (ordinary) Riemann Hypothesis and we assume for fixed $\epsilon>0$ and in the range $0<|k|\leq x$ that 
$$  \sum_{\max(0,-k)<n\leq \min(x,x-k) } \Lambda(n)\Lambda(n+k) = \mathfrak S(k) (x-|k|) +O(x^{\frac 12+\epsilon}), $$
where the singular series for prime pairs is defined by 
\begin{equation}
\label{equation singular series}
\mathfrak S(k) := \begin{cases}
2 \prod_{p\neq 2} \left( 1-\frac 1{(p-1)^2}\right) \prod_{\substack{p\mid k \\ p\neq 2}} \left( \frac{p-1}{p-2}\right) & \text{ if } k\neq 0 \text{ is even},\\
0 &\text{ if }k \text{ is odd,}
\end{cases} 
\end{equation}
then we have for $x^{\frac 12+\epsilon} \leq q \leq x^{1-\epsilon}$ the estimate
\begin{equation}
\frac{V(x;q)}x = \log q -\bigg( \gamma +\log(2\pi) + \sum_{p\mid q} \frac{\log p}{p-1} \bigg)+O(x^{-\delta(\epsilon)}).
\label{equation conjecture FG}
\end{equation} 
In the range $x^{1-\epsilon} \leq q \leq x$, their estimate takes the form $V(x;q)/x = \log q +O((\log\log q)^3)$. Finally, Vaughan \cite{V} established upper bounds on the general $k$-th moment of the error term in \eqref{equation conjecture FG} for $Q/2<q\leq Q$ in the range $x(\log x)^{-A} \leq Q \leq x$ (unconditionally) and $x^{\frac 34+\epsilon} \leq Q \leq x$ (under GRH).

%

We wish to emphasize that \eqref{equation Goldston Vaughan} gives an asymptotic result for $Q$ in the range $x^{\frac 12+\epsilon} \leq Q \leq x$, and that \textit{no asymptotic results are known in the range $Q\leq x^{\frac 12}$}, even conditionally. Moreover, \eqref{equation conditional lower bound on V(x;q)} gives the correct order of magnitude of $V(x;q)$, provided $x^{\frac 23+\epsilon} \leq q \leq x$. Even under a strong version of the Hardy-Littlewood Conjecture, the best known range of validity of \eqref{equation hooley conjecture} is $x^{\frac 12+\epsilon} \leq q \leq x$. Therefore, the behaviour of $V(x;q)$ for $q\leq x^{\frac 12}$ is a mystery and it is not clear whether an asymptotic formula such as \eqref{equation hooley conjecture} should hold in this range. As Friedlander and Goldston \cite{FG} put it, 
\begin{quote}
It may well be that these also hold for smaller $q$, but below $q=x^{\frac 12}$ we are somewhat skeptical.
\end{quote} 

In a recent paper, Keating and Rudnick \cite{KR} (see also \cite{K1,K2}) established a function field analogue of Hooley's conjecture which suggests that \eqref{equation hooley conjecture} might hold in the extended range $x^{\epsilon} \leq q\leq x$, for any fixed $\epsilon>0$. 

In the current paper we establish a probabilistic result which suggests that \eqref{equation hooley conjecture} should hold all the way down to $q=(\log\log x)^{1+\delta}$, and should not hold in the range $q\leq(\log\log x)^{1-\delta}$. 

\begin{conjecture}
Fix $\delta>0$. In the range $(\log\log x)^{1+\delta}\leq q\leq x$ we have
$$ V(x;q)\sim x\log q.$$

\label{conjecture Hooley in large range}
\end{conjecture}

To justify this conjecture, we will show in Proposition \ref{proposition H_q as sum of Y} that under GRH and a Linear Independence Hypothesis, the limiting logarithmic distribution of $\phi(q)V(x;q)/x$ coincides with the distribution of an explicit random variable $H_q$ (see Definition \ref{definition H_q}).

\subsection{Analysis of the random variable $H_q$}
We analyze the random variable $H_q$ defined in \eqref{equation definition of H}, 
 by first computing its mean and variance.
\begin{theorem}
Assuming GRH,
the random variable $H_q$ defined in \eqref{equation definition of H} satisfies
$$ \E[H_q]= \phi(q) \log q\bigg(1+O\left(\frac{\log\log q}{\log q}\right) \bigg), \hspace{.4cm} \V[H_q] = 2\phi(q) (\log q)^2\left(1 +O\left( \frac {1}{\log\log q}\right)\right).$$
Assuming moreover that $L(\frac 12,\chi)\neq 0$ (this is Chowla's Conjecture), we have that
\begin{equation}
\E[H_q]=\phi(q)\bigg( \log q -\gamma -\log (2\pi) -\sum_{p\mid q} \frac{\log p}{p-1}+O\left(\frac{(\log q)^2}{q}\right) \bigg).
\label{equation precise estimate mean}
\end{equation} 
\label{theorem first two moments}
\end{theorem}

\begin{remark}
Without assuming Chowla's Conjecture, one can obtain the estimate \eqref{equation precise estimate mean} with an additional term involving real zeros of $L(s,\chi)$ (see \eqref{equation precise estimate mean without Chowla}).
\label{remark without chowla}
\end{remark}


Hence, $H_q$ is a random variable which is concentrated about its mean $\E[H_q]\sim\phi(q)\log q$.
In light of Proposition \ref{proposition H_q as sum of Y}, this gives an intuitive reason why \eqref{equation hooley conjecture} should hold. Our main result is an estimate on the large deviations of $H_q$, which gives information on the range of validity of \eqref{equation hooley conjecture}. 

%

\begin{theorem}
\label{theorem large deviations}
Assume GRH, and let $H_q$ be the random variable defined in \eqref{equation definition of H}. If $q$ is large enough and $ (\log\log q)^2/\log q \leq \epsilon < \epsilon_0$ ($\epsilon$ can depend on $q$), then we have the following bounds on the large deviations of $H_q$:
\begin{equation}
\frac 14 \exp( -c_1 \epsilon^2 \phi(q) )  \leq \P[|H_q-\phi(q)\log q| > \epsilon \phi(q) \log q] \leq 2\exp(- c_2\epsilon^2 \phi(q) ).
\label{equation theorem large devations}
\end{equation} 
Here, $\epsilon_0$, $c_1$ and $c_2$ are absolute constants.
\end{theorem}



\begin{remark}
\label{remark extended range epsilon}
If in addition to GRH we assume that $L(\frac 12,\chi) \neq 0$ for all primitive $\chi$ and we replace $H_q-\phi(q)\log q$ with $H_q-\phi(q)\big( \log q -\gamma -\log (2\pi) -\sum_{p\mid q} \frac{\log p}{p-1}\big)$ in \eqref{equation theorem large devations}, then we can extend the range of $\epsilon$ to 
$\Psi(q) \log q/q  \leq \epsilon \leq \epsilon_0,$
where $\Psi(q)$ is any function tending to infinity with $q$ and $\epsilon_0$ is an absolute constant. 

\end{remark}

\begin{remark}
What Theorem \ref{theorem large deviations} roughly says is that the large deviations of $H_q$ are those of a normal distribution of mean $\E[H_q] \approx \phi(q)\log q$ and variance $\V[H_q] \approx 2\phi(q)(\log q)^2$ (see Theorem \ref{theorem first two moments}). Indeed, such a distribution $Z_q\sim N(\phi(q) \log q, 2\phi(q)(\log q)^2)$ satisfies
\begin{align*}
\P[|Z_q-\phi(q)\log q|> \epsilon\phi(q)\log q ] &=\P\bigg[ \frac{|Z_q-\E[Z_q]|}{\sqrt{\V[Z_q]}}> \epsilon \sqrt{\phi(q)/2}\bigg]  \\
& = \frac{2}{\sqrt{2\pi}}\int_{\epsilon \sqrt{\phi(q)/2}}^{\infty}  e^{-t^2/2}dt \sim \frac 2{\epsilon \sqrt{\pi \phi(q)}} \exp(-\epsilon^2 \phi(q)/4).
\end{align*} 
\end{remark}

\subsection{Relation between $V(x;q)$ and $H_q$}

Assuming the following linear independence hypothesis and GRH, we will show that the distribution of $H_q$ coincides with the limiting logarithmic distribution of $\phi(q)V(x;q)/x$.

\begin{hypothesis}[Linear Independence]
For any $q\geq 1$, the multiset $Z(q):=\{ \Im(\rho) \geq 0 : L(\rho,\chi)=0,\chi\bmod q, \Re(\rho)\geq \frac 12\}$, that is the set of all non-negative imaginary parts of zeros of $L(s,\chi)$ with $\chi\bmod q$ and $\Re(s)\geq \frac 12$, is linearly independent over $\mathbb Q$.
\end{hypothesis}

Hypothesis LI first appeared for $\zeta(s)$ in the work of Wintner \cite{W} on the distribution of $\psi(x)-x$, and has subsequently been used by many authors for similar purposes \cite{Ho7, Mo, Mn}. It is now a standard hypothesis in the study of prime number races \cite{RubSar, FeM, Ma, Ng1, FiMa, Lam3,Lam,Lam2, Fi1, Fi}. Partial results towards LI include the work of Martin and Ng, and Li and Radziwill in the case of Dirichlet $L$-functions \cite{MN,LR} and the work of Kowalski in the context of hyperelliptic curves over finite fields \cite{Ko}.


\begin{remark}
Hypothesis LI implies that all nontrivial zeros of $L(s,\chi)$ are simple and do not lie on the real line.
\end{remark}

\begin{proposition}
Assume GRH and LI. Then as $x\rightarrow \infty$, the limiting logarithmic distribution of $\phi(q)V(x;q)/x$ coincides with that of the random variable $H_q$.
\label{proposition H_q as sum of Y}
\end{proposition}

Going back to Theorem \ref{theorem large deviations}, we see that under GRH and LI, the probability that $V(e^y;q) \not\sim e^y\log q $, that is the probability that for a fixed $0 < \epsilon \leq \epsilon_0$ we have $ |H_q-\phi(q)\log q| > \epsilon \phi(q) \log q,  $ is at most $2\exp(-c_2 \epsilon^2 \phi(q))$. Hence we expect that this event does not happen at all in the range $y =o( \exp( c_2\epsilon^2 \phi(q)))$; this translates to the statement that $V(x;q)\sim x\log q$ in the range $(\log\log x)^{1+\delta} < q \leq x^{o(1)}$, justifying Conjecture \ref{conjecture Hooley in large range}. We will expand this argument in the concluding remarks.

\begin{remark}
It is interesting to note that the secondary term $-\gamma-\log (2\pi) - \sum_{p\mid q} \frac{\log p}{p-1}$ appearing in Theorem \ref{theorem first two moments} is identical to the secondary term appearing in \eqref{equation conjecture FG}.
Indeed, these terms come from quite different sources. The one appearing in the current paper comes from the following GRH estimate, with the $\gamma_{\chi}$ running (with multiplicity) through the imaginary parts of the non-trivial zeros of $L(s,\chi)$:
\begin{equation}
\sum_{\substack{\chi\bmod q \\ \chi\neq \chi_0}} \sum_{\gamma_{\chi}} \frac 1{\frac 14+\gamma^2} = \phi(q)\bigg( \log q -\gamma -\log (2\pi) -\sum_{p\mid q} \frac{\log p}{p-1}+O\left(\frac{(\log q)^2}{q}\right) \bigg), 
\end{equation} 
whereas the one in \cite{FG} comes from their Proposition 3, which is an estimate for the average of the Hardy-Littlewood singular series for prime pairs defined in \eqref{equation singular series}.
Their estimate takes the form
\begin{equation}
\sum_{j\leq y} \mathfrak S(jq)\left(1-\frac{j}{y} \right) = \frac{y}{2} \frac q{\phi(q)} -\frac{\log y}2 -\frac12 \bigg( \gamma + \log(2\pi) -1+\sum_{p\mid q} \frac{\log p}{p-1}\bigg) +J_{\delta}(y,q),  
\label{equation average singular series}
\end{equation} 
where $J_{\delta}(y,q)$ is an error term. Actually, this is evidence for Chowla's Conjecture, which states that $L(\frac 12,\chi)\neq 0$. Indeed, going back to Remark \ref{remark without chowla}, Theorem \ref{theorem first two moments} gives an estimate for $\E[H_q]$ in which a term depending on the real zeros of $L(s,\chi)$ appears:
\begin{equation}
\E[H_q]=\phi(q)\bigg( \log q -\gamma -\log (2\pi) -\sum_{p\mid q} \frac{\log p}{p-1}-4\sum_{\chi\neq \chi_0} z_{\chi}+O\left(\frac{(\log q)^2}{q}\right) \bigg),
\label{equation precise estimate mean without Chowla}
\end{equation} 
where $z_{\chi}$ is the order of vanishing of $L(s,\chi)$ at $s=\frac 12$. Comparing this with \eqref{equation conjecture FG}, one sees that the extra term involving $z_{\chi}$ should be of very small order, giving evidence for Chowla's Conjecture.

\end{remark}





\subsection{Method of proof and possible extensions}
The proof of Theorem \ref{theorem large deviations} is based on the analytic properties of the moment-generating function (Laplace transform) of $H_q$. Previous estimates on large deviations of error terms of prime counting functions \cite{Mo, Mn, MoOd, RubSar, Ng1, Ng, Lam} are based on an explicit formula for the moment-generating function of the associated random variable which involves infinite products of Bessel functions. For example, under GRH and LI, the moment-generating function of the distribution of the error term $e^{-y/2}(Li(e^y)-\pi(e^y))$ is given by 
$$ \mathcal L(z)=e^{z} \prod_{\gamma>0} I_0\bigg( \frac {2z}{\sqrt{\frac 14+\gamma^2}}\bigg), $$
where $I_0$ is the modified Bessel function and $\gamma$ runs through the imaginary parts of the nontrivial zeros of $\zeta(s)$.
We could not use this approach in our analysis since such a nice closed form is not known for the moment-generating function of $H_q$. Instead we use that 
$$ H_q = \sum_{\substack{\chi \in C(q) }} |Y_{\chi}|^2, $$
where $C(q)$ is a certain set of characters $\bmod q$ and the $Y_{\chi}$ are independent random variables whose real and imaginary parts have moment-generating functions which are known explicitly (see Lemma \ref{lemma moment generating of Y_chi}). From this we compute the moments of $H_q$ in terms of the moments of $|Y_{\chi}|^2$, which we then bound using complex analysis. The moment-generating functions of $\Re(Y_{\chi})$ and $\Im(Y_{\chi})$ are entire; this is a consequence of the fact that the $n$-th moments of $\Re(Y_{\chi})$ and $\Im(Y_{\chi})$ are bounded above by $n!^{\frac 12+O_q(1/\log n)}$. From this we obtain a bound for the $n$-th moment of $|Y_{\chi}|^2$ of order $n!^{1+O_q(1/\log n)}$, which we believe is best possible (such is the case with the Gaussian). Hence we deduce the existence of the moment-generating function of $H_q$ inside the circle $|z|=(c\log q)^{-1}$. 
Using this information we give bounds on large deviations of $H_q$ by using a method similar to that used to prove the Bernstein Inequalities. 

\begin{remark}
It is possible to say something about the limiting logarithmic distribution of $\phi(q)V(x;q)/x$ without the assumption LI. Indeed, one can adapt the techniques used in Section 2 of \cite{Fi} to show that this distribution has mean
$$\phi(q)\bigg( \log q -\gamma -\log (2\pi) -\sum_{p\mid q} \frac{\log p}{p-1}+O\left(\frac{(\log q)^2}{q}\right) \bigg),$$
 under GRH and the assumption that the nontrivial zeros of
$$Z_q(s) := \prod_{\substack{\chi \bmod q \\ \chi\neq \chi_0 }} L(s,\chi) $$
are simple and nonreal. This is based on the fact that the calculation of the mean in the proof of Theorem \ref{theorem first two moments} depends only on the fact that the $Z_{\gamma_{\chi}}$ have zero covariance. 
\end{remark}

\begin{remark}
One could ask if the methods of the current paper generalize to the study of the $m$-th moment
$$ H^{(m)}(x;q):=\sum_{\substack{ a \bmod q \\ (a,q)=1}} \left( \psi(x;q,a)-\frac{\psi(x,\chi_0)}{\phi(q)} \right)^m. $$
The first problem of our approach to this problem for $m\geq 3$ is the lack of a nice formula such as \eqref{equation H n terms of characters}, and hence it seems hard to turn this question into a question about sums of independent random variables. Another serious problem is that the method of moments would fail for $m= 3$. Indeed the formula analogous to \eqref{equation H n terms of characters} for $m=3$ contains terms of the form $\psi(x,\chi)^3$ with cubic characters $\chi$, and the moments of these terms grow too fast for the moment method to be applied. This is analogous to the fact that if $Z$ is a standard Gaussian, then the distribution of $Z^3$ is indeterminate. Indeed, Berg \cite{Be} has given an explicit infinite family of distinct random variables whose moments coincide with those of $Z^3$. Specifically, if $|r|\leq \frac 12$ is any real number, then the random variable whose probability density is
$$ f(t) = \frac 1{3\sqrt{\pi}} |t|^{-\frac 23} e^{-|t|^{\frac 23}}  \left( 1+r\cos(\sqrt 3 |t|^{\frac 23}) -\sqrt 3 \sin(\sqrt 3 |t|^{\frac 23}) \right) $$
has exactly the same moments as $Z^3$.
\end{remark}

\section{Link with random variables}

We now relate the study of $V(x;q)$ to that of the random variable $H_q$.
As we will use orthogonality relations, it will be useful to treat real and complex characters seperately. Throughout the paper, $C(q)$ will denote a fixed subset of the Dirichlet characters modulo $q$ such that $C(q)$ contains each non-principal real character, and contains exactly one of $\chi$ or $\overline{\chi}$ for complex characters $\chi$. 

\begin{definition}
\label{definition H_q}
We define the random variables $Z_{\chi;\gamma_{\chi}}=Z_{\gamma_{\chi}}$ to be i.i.d. random variables uniformly distributed on the unit circle in $\mathbb C$, where $\chi$ runs over the set $C(q)$ and $\gamma_{\chi}$ runs over the imaginary parts of the nontrivial zeros of $L(s,\chi)$ (in the case $\chi$ is real, only positive imaginary parts will be considered). We also define
\begin{equation}
H_q:= \sum_{\chi \in C(q)} |Y_{\chi}|^2,
\label{equation definition of H}
\end{equation} 
where
\begin{equation}
Y_{\chi}:= 
\begin{cases}
2\sum_{\gamma_{\chi}>0}  \sqrt{\frac{m_{{\gamma_{\chi}}}}{\frac 14+\gamma_{\chi}^2}} \Re (Z_{\gamma_{\chi}}) & \text{ if } \chi \text{ is real,} \\
\sqrt 2\sum_{\gamma_{\chi}\neq 0} \sqrt{\frac{m_{{\gamma_{\chi}}}}{\frac 14+\gamma_{\chi}^2}} Z_{\gamma_{\chi}}& \text{ if } \chi \text{ is complex.}
\end{cases}
\label{equation definition of Y_chi}
\end{equation} 
Here $m_{\gamma_{\chi}}$ denotes the multiplicity of $\rho_{\chi}=\frac 12+i\gamma_{\chi}$, and the sums over zeros are counted without multiplicity.
\end{definition}

\begin{remark}
It might be preferable to use the notation $Z_{\chi;\gamma_{\chi}}$ rather than $Z_{\gamma_{\chi}}$, since in the way we define these random variables we stipulate that if $\chi\neq \chi'$, then $Z_{\chi,\gamma_{\chi}}$ and $Z_{\chi',\gamma_{\chi'}}$ are independent, whatever $\gamma_{\chi}$ and $\gamma_{\chi'}$ are. We  will keep the notation $Z_{\gamma_{\chi}}$ to be more concise. Note also that LI implies that $m_{\gamma_{\chi}}=1$.
\end{remark}

\begin{remark}
The collection $\{|Y_{\chi}|^2:\chi \in C(q)\}$ is independent. This fact will be useful when computing the moments of $H_q$.
\end{remark}

We now relate $V(x;q)$ and the random variable $H_q$.

\begin{proof}[Proof of Proposition \ref{proposition H_q as sum of Y}]

Using orthogonality relations, we compute
\begin{align}
\begin{split}
V(x;q) & = \sum_{\substack{a \bmod q \\ (a,q)=1}} \bigg| \frac 1{\phi(q)}\sum_{\chi \neq \chi_0} \overline{\chi}(a)\psi(x,\chi) \bigg|^2\\
&= \frac 1{\phi(q)^2}\sum_{\substack{\chi_1,\chi_2\bmod q \\ \chi_1,\chi_2\neq \chi_0}} \psi(x,\chi_1)\psi(x,\overline{\chi_2})\sum_{\substack{ a \bmod q \\ (a,q)=1}}\overline{\chi_1}(a)\chi_2(a)\\
&= \frac 1{\phi(q)} \sum_{\substack{\chi\bmod q \\ \chi \neq \chi_0}} \left|  \psi(x,\chi)  \right|^2.
\end{split}
\label{equation H n terms of characters}
\end{align}
Applying GRH to the explicit formula, we obtain that
$$ V(x;q) = \frac x{\phi(q)} \sum_{\substack{\chi\bmod q \\ \chi \neq \chi_0}} \bigg|  \sum_{\gamma_{\chi}} \frac{e^{i\gamma_{\chi} \log x}}{\frac 12+i\gamma_{\chi}}  \bigg|^2+O(\sqrt x (\log x)^3), $$
so 
\begin{equation}
 \phi(q)e^{-y}V(e^y;q) = \sum_{\substack{\chi\bmod q \\ \chi \neq \chi_0}} \bigg|  \sum_{\gamma_{\chi}} \frac{e^{i\gamma_{\chi} y}}{\frac 12+i\gamma_{\chi}}  \bigg|^2+O(\phi(q)e^{-\frac y2}y^3).
 \label{equation expression for V(x;q)}
\end{equation} 
Using the fact that $\overline{L(s,\chi)} = L(\overline s,\overline{\chi})$ and that real nontrivial zeros do not exist under LI, we transform the sum over zeros as follows:
\begin{align*}
\sum_{\substack{\chi\bmod q \\ \chi \neq \chi_0}} \bigg|&  \sum_{\gamma_{\chi}} \frac{e^{i\gamma_{\chi} y}}{\frac 12+i\gamma_{\chi}}  \bigg|^2 = \sum_{\substack{\chi \neq \chi_0 \\ \chi\text{ real}}} \bigg|  \sum_{\gamma_{\chi}>0} \bigg(\frac{e^{i\gamma_{\chi} y}}{\frac 12+i\gamma_{\chi}}  +\frac{e^{-i\gamma_{\chi} y}}{\frac 12-i\gamma_{\chi}} \bigg)\bigg|^2 + \sum_{\substack{\chi \in C(q) \\ \chi\text{ complex}}} \bigg|  \sum_{\gamma_{\chi} \neq 0} \frac{e^{i\gamma_{\chi} y}}{\frac 12+i\gamma_{\chi}} \bigg|^2\\ &\hspace{3cm}+\sum_{\substack{\chi \in C(q) \\ \chi\text{ complex}}} \bigg|  \sum_{\gamma_{\overline{\chi}}\neq 0} \frac{e^{i\gamma_{\overline{\chi}} y}}{\frac 12+i\gamma_{\overline{\chi}}} \bigg|^2\\
&=  \sum_{\substack{\chi \bmod q \\ \chi \neq \chi_0 \\ \chi\text{ real}}} \bigg|  2\Re \sum_{\gamma_{\chi}>0} \frac{e^{i\gamma_{\chi} y}}{\frac 12+i\gamma_{\chi}}  \bigg|^2+\sum_{\substack{\chi \in C(q) \\ \chi\text{ complex}}} \bigg|  \sum_{\gamma_{\chi}\neq 0} \frac{e^{i\gamma_{\chi} y}}{\frac 12+i\gamma_{\chi}} \bigg|^2 +\sum_{\substack{\chi \in C(q) \\ \chi\text{ complex}}} \bigg|  \sum_{\gamma_{\chi}\neq 0} \frac{e^{-i\gamma_{\chi} y}}{\frac 12-i\gamma_{\chi}} \bigg|^2 \\
&= 4\sum_{\substack{\chi \bmod q \\ \chi \neq \chi_0 \\ \chi\text{ real}}} \bigg( \sum_{\gamma_{\chi}>0} \Re \frac{e^{i\gamma_{\chi} y}}{\frac 12+i\gamma_{\chi}}  \bigg)^2+ 2 \sum_{\substack{\chi \in C(q) \\ \chi\text{ complex}}} \bigg|  \sum_{\gamma_{\chi}\neq 0} \frac{e^{i\gamma_{\chi} y}}{\frac 12+i\gamma_{\chi}} \bigg|^2.
\end{align*}
%
Now, $|\frac 12+i\gamma_{\chi}|=\sqrt{\frac 14+\gamma_{\chi}^2}$, and under LI\footnote{Note that if $\chi$ is complex, then LI implies that the set of imaginary parts of nontrivial zeros of $L(s,\chi)$ is linearly independent. This follows from the fact that the zeros of positive imaginary part of both $L(s,\chi)$ and $L(s,\overline{\chi})$ appear in the set $Z(q)$ in the statement of LI.}, if we order the $\gamma_{\chi}$ appearing in the above sums by size, then for every fixed $k$ the vector $(e^{i\gamma_1 y},...,e^{i\gamma_{k} y}) \subset \mathbb T^k$
becomes equidistributed as $y\rightarrow \infty$ by the Kronecker-Weyl Theorem. The assertion follows similarly as in Section 2 of \cite{RubSar} or Proposition 2.3 of \cite{FiMa}.



\end{proof}


\section{The first two moments of $H_q$}
In order to prove Theorem \ref{theorem first two moments} we need several lemmas.

\begin{lemma}
\label{lemma b(chi)}
Assume GRH and let $\chi\neq \chi_0$ be a character $\bmod q$. Then, letting $\gamma_{\chi}$ run over the imaginary parts of the nontrivial zeros of $L(s,\chi)$ we have
\begin{align*}
\sum_{\gamma_{\chi}} \frac {m_{\gamma_{\chi}}}{\frac 14+\gamma_{\chi}^2} &= \log \frac{q^*}{\pi} -\gamma-(1+\chi(-1)) \log 2 + 2\Re\frac{L'(1,\chi^*)}{L(1,\chi^*)}  \\
&= \log q^* +O(\log\log q^*),
\end{align*}
where $m_{\gamma_{\chi}}$ denotes the multiplicity of $\rho_{\chi}=\frac 12+i\gamma_{\chi}$, and the sum is counted without multiplicity.
\end{lemma}
\begin{proof}
The first equality is Lemma 3.5 of \cite{FiMa}. The second follows from applying Littlewood's GRH bound $L'(1,\chi)/L(1,\chi) \ll \log \log q^*$ (see \cite{Li}) to the first.
\end{proof}
We will need a bound for the average of $2\Re(L'(1,\chi^*)/L(1,\chi^*)) $ over $\chi\neq \chi_0$.
\begin{lemma}
\label{lemma analytic term}
Under GRH, the following holds:
$$ \sum_{\substack{\chi\bmod q \\ \chi \neq \chi_0}} \frac{L'(1,\chi^*)}{L(1,\chi^*)} \ll \phi(q) \frac{(\log q)^2}q. $$
\end{lemma}
\begin{proof}
We have 
\begin{equation}
 \sum_{\substack{\chi\bmod q \\ \chi \neq \chi_0}} \frac{L'(1,\chi^*)}{L(1,\chi^*)} = \sum_{\substack{\chi\bmod q \\ \chi \neq \chi_0}}\sum_n \frac{\chi^*(n) \Lambda(n)}{n} = \sum_n \frac{\Lambda(n)}n \sum_{\substack{\chi\bmod q \\ \chi \neq \chi_0}} \chi^*(n). 
 \label{equation 1 lemma M}
\end{equation}
Now, taking $r=1$ in Proposition 3.4 of \cite{FiMa} shows that for $e\geq 1$,
$$ \sum_{\chi \bmod q} (\chi^*(p^e)-\chi(p^e)) = \begin{cases}
\phi(q/p^{\nu}) &\text{ if } p^{\nu} \parallel q, \nu \geq 1 \text{ and } p^e \equiv 1 \bmod q/p^{\nu},\\
0 &\text{ otherwise.}
\end{cases} $$
Therefore, denoting by $e(q;p)$ the least $e\geq 1$ such that $p^e \equiv 1 \bmod q/p^{\nu}$ (note that $p^{e(q;p)} \geq q/p^{\nu}$) we have
\begin{align*}
\sum_n \frac{\Lambda(n)}n \sum_{\chi\bmod q} (\chi^*(n)-\chi(n)) &=\sum_{\substack{p^{\nu}\parallel q \\ p^e \equiv 1 \bmod q/p^{\nu} \\\nu, e\geq 1}} \frac{\log p}{p^e} \phi(q/p^{\nu}) = \sum_{\substack{p^{\nu}\parallel q \\ \nu\geq 1} } \phi(q/p^{\nu}) \log p \sum_{\substack{ e\geq 1 \\ p^e \equiv 1 \bmod q/p^{\nu}}} \frac 1{p^e} 
\\  &= \sum_{\substack{p^{\nu}\parallel q \\ \nu\geq 1} }\phi(q/p^{\nu}) \log p \frac 1{p^{e(q;p)} (1-p^{-e(q;p)} )}\ll \sum_{\substack{p^{\nu}\parallel q \\ \nu\geq 1} } \frac{\phi(q/p^{\nu}) \log p}{p^{e(q;p)}} \\
&\leq \sum_{\substack{p^{\nu}\parallel q \\ \nu\geq 1} } \frac{\phi(q/p^{\nu}) \log p}{q/p^{\nu}} \ll \log q.
\end{align*}  
Moreover,
$$ \sum_{n} \frac{\Lambda(n)}n (\chi_0^*(n)-\chi_0(n)) = \sum_{\substack{p^{\nu}\parallel q \\ \nu\geq 1} } \log p \sum_{1\leq e\leq \nu} \frac 1{p^e} \ll  \sum_{p \mid  q} \frac {\log p}p \ll \log\log q. $$

Hence, \eqref{equation 1 lemma M} becomes
$$ \sum_{\substack{\chi\bmod q \\ \chi \neq \chi_0}} \frac{L'(1,\chi^*)}{L(1,\chi^*)}  = \sum_n \frac{\Lambda(n)}n  \sum_{\substack{\chi\bmod q \\ \chi \neq \chi_0}} \chi(n)  + O(\log q)  = \bigg(  \phi(q) \sum_{n\equiv 1 \bmod q} -\sum_n \bigg) \frac{\Lambda(n)}{n}+O(\log q),$$
where term on the right hand side should be interpreted as the limit of the truncated sums. We first treat the values of $n$ for which $n> q^2$:

\begin{align*}
 \bigg(  \phi(q) \sum_{\substack{n\equiv 1 \bmod q \\ n>q^2}} -\sum_{n>q^2} \bigg) \frac{\Lambda(n)}{n}  &= \int_{q^2}^{\infty} \frac{d(\phi(q) \psi(t;q,1)-\psi(t))}{t} \\
& =\frac{\phi(q) \psi(t;q,1)-\psi(t)}t \bigg|_{q^2}^{\infty}+ \int_{q^2}^{\infty} \frac{\phi(q) \psi(t;q,1)-\psi(t)}{t^2} dt \\
&\ll \phi(q) (q^2)^{-\frac 12} (\log q)^2 + \phi(q) \int_{q^2}^{\infty} \frac{(\log(t^2))^2}{t^{\frac 32}} dt \ll \phi(q)\frac{(\log q)^2}q,
\end{align*} 
by GRH. As for the values $n\leq q^2$, we have the following elementary estimates:
$$ \phi(q) \sum_{\substack{ n\equiv 1 \bmod q \\ n\leq q^2}} \frac{\Lambda(n)}{n}  \leq 2\phi(q) \log q \sum_{1\leq j \leq q} \frac 1{qj+1}   \ll \phi(q)\frac{(\log q)^2}q,  $$
$$ \sum_{n\leq q^2}\frac{\Lambda(n)}{n} \ll \log q. $$
We conclude that 
$$ \sum_{\substack{\chi\bmod q \\ \chi \neq \chi_0}} \frac{L'(1,\chi^*)}{L(1,\chi^*)}  \ll \log q +\phi(q)\frac{(\log q)^2}q ,$$
and the result follows from the bound $\phi(q) \gg q/\log\log q$.
\end{proof}

\begin{lemma} For any $q\geq 3$,
\begin{equation}
\sum_{\substack{\chi\bmod q \\ \chi \neq \chi_0}} \log q^*  = \phi(q) \bigg( \log q - \sum_{p\mid q} \frac{\log p}{p-1}\bigg)
= \phi(q)\left( \log q + O(\log\log q)\right),  
\label{equation first statement of lemma conductors}
\end{equation} 
$$ \sum_{\substack{\chi\bmod q \\ \chi \neq \chi_0}} (\log q^*)^2  = \phi(q) (\log q)^2 \left( 1 + O\left( \frac {\log\log q}{\log q} \right)\right).  $$
\label{lemma FiMa}
\end{lemma}
\begin{proof}
The first statement is Proposition 3.3 of \cite{FiMa}. 
As for the second, we adapt Proposition 3.3 of \cite{FiMa}. The arithmetical function
$ \Lambda_2(n) := \sum_{d\mid n} (\log d)^2 \mu(n/d)$
is supported on integers having at most two prime factors, and takes the following values: 
\begin{equation}
 \Lambda_2(p^e) = (2e-1) (\log p)^2,\hspace{1cm}  \Lambda_2(p_1^{e_1} p_2^{e_2} ) = 2\log p_1\log p_2. 
 \label{equation values of lambda 2}
\end{equation}
Following (3.2) of \cite{FiMa}, we compute
\begin{align*}
\sum_{d\mid q} \phi(d) \Lambda_2(q/d) &= \sum_{d\mid q} \bigg(\sum_{\chi \bmod d} 1\bigg) \Lambda_2(q/d) \\
& = \sum_{\chi \bmod q} \sum_{\substack{d \mid q \\ q^* \mid d}} \Lambda_2(q/d) \\
&= \sum_{\chi \bmod q} \sum_{\ell \mid q/q^*} \Lambda_2(\ell) \\
&= \sum_{\chi \bmod q} \big(\log \frac{q}{q^*}\big)^2 \\
&= \phi(q) (\log q)^2 -2 \log q\sum_{\chi \bmod q} \log q^* +\sum_{\chi \bmod q} (\log q^*)^2.
\end{align*}
Combining this with \eqref{equation first statement of lemma conductors} shows that
$$ \sum_{\chi \bmod q} (\log q^*)^2= \phi(q)(\log q)^2 + O\bigg( \phi(q) \log q\log\log q +  \sum_{d\mid q} \phi(d) \Lambda_2(q/d)  \bigg), $$
and so the last step is to show that $\sum_{d\mid q} \phi(d) \Lambda_2(q/d)\ll \phi(q) \log q\log\log q$. Arguing as in Lemma 3.2 of \cite{FiMa} and using \eqref{equation values of lambda 2} we compute 

\begin{align*}
\sum_{d\mid q} \phi(d) \Lambda_2(q/d) &= \sum_{p^r \parallel q} \sum_{k=0}^{r-1} \Lambda_2(p^{r-k}) \phi(q/p^{r-k}) \\\hspace{1cm} & \hspace{1cm}+ \sum_{\substack{p_1^{r_1} \parallel q, p_2^{r_2} \parallel q \\ p_1<p_2 } } \sum_{\substack{ 0\leq k_1\leq r_1-1 \\ 0\leq k_2\leq r_2-1}} \Lambda_2(p_1^{r_1-k_1} p_2^{r_2-k_2}) \phi(q/(p_1^{r_1-k_1}p_2^{r_2-k_2})) \\
&= \sum_{p^r \parallel q} \phi(q/p^r) \sum_{k=0}^{r-1} (2r-2k-1) (\log p)^2 \phi(p^k)\\
&\hspace{1cm} +2 \sum_{\substack{p_1^{r_1} \parallel q, p_2^{r_2} \parallel q \\ p_1<p_2 } } \phi(q/(p_1^{r_1}p_2^{r_2})) \sum_{\substack{ 0\leq k_1\leq r_1-1 \\ 0\leq k_2\leq r_2-1}}  \log p_1 \log p_2 \phi(p_1^{k_1}) \phi(p_2^{k_2}) \\
&=  \sum_{p^r \parallel q} \phi(q/p^r)  (\log p)^2  \left[\left( 1-\frac 1p \right)\frac{p^{r+1}+p^r-p(2r+1)+2r-1 }{(p-1)^2} +\frac{2r-1}p \right] \\
&\hspace{1cm} +2 \sum_{\substack{p_1^{r_1} \parallel q, p_2^{r_2} \parallel q \\ p_1<p_2 } } \phi(q/(p_1^{r_1}p_2^{r_2})) \log p_1 \log p_2 p_1^{r_1-1}p_2^{r_2-1}\\
&\ll \sum_{p^r \parallel q} \phi(q/p^r)  (\log p)^2 \left(p^{r-1} +\frac r{p^2}\right) + \phi(q) \left(\sum_{p\mid q} \frac{\log p}{p} \right)^2 \ll \phi(q) \log q,
\end{align*}
completing the proof.

%
%

\end{proof}

For a real-valued random variable $W$, we will use the following notation for its moment-generating function:
$$ \mathcal L_W(z) := \E[e^{z W}]. $$ 
\begin{lemma}
\label{lemma moment generating of Y_chi}
Assume GRH and define $Y_{\chi}$ as in \eqref{equation definition of Y_chi}. Then for real characters $\chi$, the moment-generating function of $Y_{\chi}$ is an even function of $z$ given by
\begin{equation}
  \mathcal L_{Y_{\chi}}(z):= \prod_{\gamma_{\chi}>0} I_0\bigg( 2z \sqrt{\frac{m_{\gamma_{\chi}}}{\frac 14+\gamma_{\chi}^2}} \bigg),
   \label{equation moment generating of Y_chi}
\end{equation} 
where $\gamma_{\chi}$ runs over the imaginary parts of the nontrivial zeros of $L(s,\chi)$, $m_{\gamma_{\chi}}$ denotes the multiplicity of $\rho_{\chi}= \frac 12+i\gamma_{\chi}$ and $I_0$ is the modified Bessel of the first kind:
$$ I_0(z) = \sum_{n=0}^{\infty} \frac 1{n!^2} \left( \frac z2 \right)^{2n}. $$ 
If $\chi$ is complex, then
\begin{equation}
  \mathcal L_{\Re (Y_{\chi})}(z) =\mathcal L_{\Im (Y_{\chi})}(z) = \prod_{\gamma_{\chi}\neq 0}I_0\bigg( 2z \sqrt{\frac{m_{\gamma_{\chi}}}{\frac 14+\gamma_{\chi}^2}} \bigg).
   \label{equation moment generating of Y_chi complex}
\end{equation}

\end{lemma}
\begin{proof}
First note that the $Z_{\gamma_{\chi}}$ appearing in \eqref{equation definition of Y_chi} are independent, and thus if $\chi$ is real, then
$$\E[e^{zY_{\chi}}] = \prod_{\gamma_{\chi}>0} \E[e^{2(m_{\gamma_{\chi}}/(\frac 14+\gamma_{\chi}^2))^{\frac 12}z\Re(Z_{\gamma_{\chi}})}]=  \prod_{\gamma_{\chi}>0} \mathcal L_{\Re(Z_{\gamma_{\chi}})} \bigg( 2z \sqrt{\frac{m_{\gamma_{\chi}}}{\frac 14+\gamma_{\chi}^2}} \bigg). $$
The proof of \eqref{equation moment generating of Y_chi} follows since the moment-generating function of $\Re(Z_{\gamma_{\chi}})$ is easily computed using the following integral representation of the Bessel $I_0$ function:
$$I_0(z) = \frac 1{\pi} \int_0^{\pi} e^{z \cos \theta} d\theta. $$
(See Proposition 2.13 of \cite{FiMa} for a similar derivation of the characteristic function of $Y_{\chi}$.)
The proof of \eqref{equation moment generating of Y_chi complex} is similar.
\end{proof}

We are now ready to prove Theorem \ref{theorem first two moments}.
\begin{proof}[Proof of Theorem \ref{theorem first two moments}]
We start with the mean, which by \eqref{equation definition of H} equals 
\begin{equation}
\E[H_q] = \sum_{\substack{\chi\in C(q)}} \E[|Y_{\chi}|^2] = \sum_{\substack{\chi\in C(q) }} \E[\Re(Y_{\chi})^2+\Im(Y_{\chi})^2] =  \sum_{\substack{\chi\in C(q) }}( \V[\Re(Y_{\chi})]+\V[\Im(Y_{\chi})]),
\label{equation first lemma variance}
\end{equation} 
since we easily get from \eqref{equation definition of Y_chi} that $\E[\Re(Y_{\chi})]=\E[\Im(Y_{\chi})]=0$. Moreover, since the random variables $Z_{\gamma_{\chi}}$ appearing in \eqref{equation definition of Y_chi} are independent and since one easily computes that $\V[\Re(Z_{\gamma_{\chi}})]=\frac 12$, we have for real $\chi$ that
\begin{equation}
\V[Y_{\chi}] = 4\sum_{\gamma_{\chi}>0} \frac{m_{\gamma_{\chi}}}{\frac 14+\gamma_{\chi}^2}\V[\Re(Z_{\gamma_{\chi}})]= \sum_{\gamma_{\chi}} \frac{m_{\gamma_{\chi}}}{\frac 14+\gamma_{\chi}^2}-4z_{\chi}, 
 \label{equation 1 lemma first two moments}
\end{equation} 
where $z_{\chi}$ is the order of vanishing of $L(s,\chi)$ at $s=\frac 12$ and the sum over $\gamma_{\chi}$ is counted without multiplicity. As for complex $\chi$, we have
\begin{equation}
\V[\Re(Y_{\chi})]=\V[\Im(Y_{\chi})] = 2\sum_{\gamma_{\chi}\neq 0} \frac{m_{\gamma_{\chi}}}{\frac 14+\gamma_{\chi}^2}\V[\Re(Z_{\gamma_{\chi}})]= \sum_{\gamma_{\chi}} \frac{m_{\gamma_{\chi}}}{\frac 14+\gamma_{\chi}^2}-4z_{\chi}. 
 \label{equation 1 lemma first two moments complex}
\end{equation} 
Hence, combining \eqref{equation first lemma variance}, \eqref{equation 1 lemma first two moments} and \eqref{equation 1 lemma first two moments complex}, we have that
\begin{align}
 \E[H_q] 
& = \sum_{\substack{\chi \bmod q \\ \chi \neq \chi_0 \\ \chi \text{ real}}} \sum_{\gamma_{\chi}} \frac{m_{\gamma_{\chi}}}{\frac 14+\gamma_{\chi}^2} -4 \sum_{\substack{\chi \bmod q \\ \chi \neq \chi_0 \\ \chi \text{ real}}} z_{\chi}+2\sum_{\substack{\chi \in C(q) \\ \chi \text{ complex}}} \sum_{\gamma_{\chi}} \frac{m_{\gamma_{\chi}}}{\frac 14+\gamma_{\chi}^2}  -8 \sum_{\substack{\chi \in C(q) \\ \chi \text{ complex}}} z_{\chi} \notag\\
&= \sum_{\substack{\chi\bmod q \\ \chi \neq \chi_0}} \sum_{\gamma_{\chi}} \frac{m_{\gamma_{\chi}}}{\frac 14+\gamma_{\chi}^2}-4 \sum_{\substack{\chi \bmod q \\\chi \neq \chi_0}} z_{\chi},
 \label{equation 1.5 lemma first two moments}
\end{align}
by definition of $C(q)$ and by the fact that $\overline{L(s,\chi)}=L(\overline s, \overline{\chi})$. 
Taking 
$$f(x):= 
\begin{cases}
1-|x| & \text{ if } |x|\leq 1 \\
0 &\text{ otherwise}
\end{cases}$$ 
in Theorem 1.3 of \cite{FiMi} (This also follows from Theorem 1.4 of \cite{GJMMNPP}) shows that 
\begin{equation}
\sum_{\substack{\chi\bmod q \\ \chi \neq \chi_0}} z_{\chi}\ll \phi(q).
\label{equation central point}
\end{equation}
(Since we are assuming GRH, the sum of the orders of vanishing at the central point is bounded above by a constant times the $1$-level density of low-lying zeros for any nonnegative test function which does not vanish at $0$.) The upper bound and the first estimate for $\E[H_q]$ follow from combining \eqref{equation 1.5 lemma first two moments} with Lemmas \ref{lemma b(chi)} and \ref{lemma FiMa}. As for the second (note that we are now assuming Chowla's Conjecture, so $z_{\chi}=0$), we combine \eqref{equation 1.5 lemma first two moments} with the exact formula in Lemma \ref{lemma b(chi)} to obtain
\begin{align*}
 \E[H_q]  &= \sum_{\substack{\chi\bmod q \\ \chi \neq \chi_0}} \bigg( \log \frac{q^*}{\pi} -\gamma-(1+\chi(-1)) \log 2 + 2\Re\frac{L'(1,\chi^*)}{L(1,\chi^*)}  \bigg) \\
  &=    \phi(q) \bigg( \log q - \gamma -\log(2\pi) - \sum_{p\mid q} \frac{\log p}{p-1} \bigg) +O(1) + 2\Re \sum_{\substack{\chi\bmod q \\ \chi \neq \chi_0}} \frac{L'(1,\chi)}{L(1,\chi)},
\end{align*}
by Lemma \ref{lemma FiMa}. The desired estimate follows by applying Lemma \ref{lemma analytic term}.

For $\V[H_q]$, we have by \eqref{equation definition of H} and by the independence of the $|Y_{\chi}|^2$ that
\begin{equation}
\label{equation 2 lemma first two moments}
 \V[H_q] = \sum_{\chi \in C(q)} \V[|Y_{\chi}|^2] =\sum_{\chi \in C(q)} (\E[|Y_{\chi}|^4]-\E[|Y_{\chi}|^2]^2).
\end{equation}  
If $\chi$ is real, then the moments of $Y_{\chi}$ can be extracted from its moment-generating function which we obtained in Lemma \ref{lemma moment generating of Y_chi} (note that this function is even):
$$ 1+ \frac{z^2}{2!}\E[Y_{\chi}^2]+  \frac{z^4}{4!}\E[Y_{\chi}^4] + \dots = \prod_{\gamma_{\chi}>0} I_0\bigg( 2z \sqrt{\frac{m_{\gamma_{\chi}}}{\frac 14+\gamma_{\chi}^2}} \bigg) = \prod_{\gamma_{\chi}>0}  \left( 1+\frac{z^2m_{\gamma_{\chi}}}{\frac 14+\gamma_{\chi}^2} + \frac{z^4m_{\gamma_{\chi}}^2}{4\left( \frac 14 +\gamma_{\chi}^2\right)^2}+\dots \right)   $$
Hence,
\begin{align}
\begin{split}
\E[Y_{\chi}^4] &= 4! \Bigg(  \sum_{\gamma_{\chi} >0} \frac {m_{\gamma_{\chi}}^2}{4\left( \frac 14 +\gamma_{\chi}^2\right)^2} + \sum_{\substack{\gamma_{\chi}>\lambda_{\chi} >0  }} \frac {m_{\gamma_{\chi}}m_{\lambda_{\chi}}}{(\frac 14+\gamma_{\chi}^2)(\frac 14+\lambda_{\chi}^2)} \Bigg) \\
&=  4! \Bigg( \frac 12 \Bigg( \sum_{\substack{\gamma_{\chi} >0}} \frac {m_{\gamma_{\chi}}}{\frac 14+\gamma_{\chi}^2} \Bigg)^2 - \frac 34\sum_{\gamma_{\chi} >0} \frac {m_{\gamma_{\chi}}^2}{\left( \frac 14 +\gamma_{\chi}^2\right)^2} \Bigg) \\
 & = 3 (\log q^*)^2 +O\left( \frac{(\log q^*)^2}{\log\log q^*}\right) +O( z_{\chi}^2),  
\end{split}
\label{equation fourth moment of Y_chi}
\end{align}
by Lemma \ref{lemma b(chi)}. Here we used that
$$ \sum_{\gamma_{\chi} >0} \frac {m_{\gamma_{\chi}}^2}{\left( \frac 14 +\gamma_{\chi}^2\right)^2} =O\left( \frac{(\log q^*)^2}{\log\log q^*}\right),$$
which follows from the GRH bound $m_{\gamma_{\chi}} \ll \log (q^* (\gamma_{\chi}+1))/\log\log (q^* (\gamma_{\chi}+3))$ (see Theorem 6 of \cite{Sel}). Note that this error term can be replaced by $O(\log q^* \log\log q^*)$ if we assume that the zeros of $L(s,\chi)$ are simple.


As for complex characters $\chi$, we have by the definition of $Y_{\chi}$ that summing over the zeros of $L(s,\chi)$,
$$ \E[|Y_{\chi}|^4] = 4 \sum_{\gamma_1,\gamma_2,\gamma_3,\gamma_4\neq 0} \frac{\sqrt{m_{\gamma_1}m_{\gamma_2}m_{\gamma_3}m_{\gamma_4}}\E[Z_{\gamma_1}Z_{\gamma_2}\overline{Z_{\gamma_3}}\overline{Z_{\gamma_4}}]}{(\frac 12+\gamma_1^2)^{\frac 12}(\frac 12+\gamma_2^2)^{\frac 12}(\frac 12+\gamma_3^2)^{\frac 12}(\frac 12+\gamma_4^2)^{\frac 12}}.$$
Moreover, by independence of the $Z_{\gamma}$ and since $\E[Z_{\gamma}]=0$ and $|Z_{\gamma}|=1$, we have that
$$\E[Z_{\gamma_1}Z_{\gamma_2}\overline{Z_{\gamma_3}}\overline{Z_{\gamma_4}}] = \begin{cases}
1 & \text{ if } \gamma_1=\gamma_2=\gamma_3=\gamma_4, \\
0 & \text{ if exactly three of the } \gamma_i \text{ are equal}, \\
1 & \text{ if } \gamma_1=\gamma_3 \neq \gamma_2=\gamma_4 \text{ or } \gamma_1=\gamma_4 \neq \gamma_2=\gamma_3, \\
0 & \text{ if } \gamma_1=\gamma_2 \neq \gamma_3=\gamma_4, \\
0& \text{ otherwise.} \\
\end{cases}  $$
Hence,
\begin{equation}
 \E[|Y_{\chi}|^4] = 4 \sum_{\gamma_1 \neq 0} \frac{m_{\gamma_1}^2}{(\frac 12+\gamma_1^2)^{2}}+ 8\sum_{ \substack{\gamma_1, \gamma_2 \neq 0\\ \gamma_1\neq \gamma_2}} \frac{m_{\gamma_1}m_{\gamma_2}}{(\frac 12+\gamma_1^2)(\frac 12+\gamma_2^2)}=8 (\log q^*)^2 +O\left(z_{\chi}^2+\frac{(\log q^*)^2}{\log\log q^*}\right)
 \label{equation 4th moment complex chi}
\end{equation}
by Theorem 6 of \cite{Sel}, Lemma \ref{lemma b(chi)} and the Riemann-von Mangoldt Formula. Combining this with \eqref{equation 2 lemma first two moments} and the previous calculation of $\E[|Y_{\chi}|^2]$ in \eqref{equation first lemma variance} and \eqref{equation 1 lemma first two moments complex} we obtain that
\begin{align*}
 \V[H_q] &=\sum_{\substack{ \chi \bmod q \\ \chi \neq \chi_0 \\ \chi \text{ real}}} \left(3 (\log q^*)^2 - (\log q^*)^2 + O\left(z_{\chi}^2+\frac{(\log q^*)^2}{\log\log q^*}\right) \right)   \\
 &+ \sum_{\substack{\chi \in C(q) \\ \chi \text{ complex}}} \left( 8 (\log q^*)^2 - 4 (\log q^*)^2  + O\left(z_{\chi}^2+\frac{(\log q^*)^2}{\log\log q^*}\right)\right)\\
 &= \sum_{\substack{\chi\bmod q \\ \chi \neq \chi_0}} 2 (\log q^*)^2 + O\left(\phi(q)\frac{(\log q)^2}{\log\log q}\right),
\end{align*}
since 
$$ \sum_{\substack{\chi\bmod q \\ \chi \neq \chi_0}} z_{\chi}^2 \leq (\max_{\chi} z_{\chi}) \sum_{\substack{\chi\bmod q \\ \chi \neq \chi_0}} z_{\chi} \ll \frac{\log q}{\log\log q} \cdot \phi(q) $$
by Theorem 6 of \cite{Sel} and \eqref{equation central point}.
The result follows from Lemma \ref{lemma FiMa}.
\end{proof}

\section{Large deviations of $H_q$}

One would like to apply the existing results on large deviations such as the Montgomery-Odlyzko bounds \cite{MoOd} to our question. Expanding the square in \eqref{equation definition of H} we obtain
\begin{align*}
H_q&=4 \sum_{\substack{\chi \bmod q \\ \chi \neq \chi_0 \\ \chi \text{ real}}} \sum_{\gamma_{\chi}>0} \frac{m_{\gamma_{\chi}}\Re(Z_{\gamma_{\chi}})^2}{ \frac 14+\gamma_{\chi}^2} + 8 \sum_{\substack{\chi \bmod q \\ \chi \neq \chi_0 \\ \chi \text{ real}}} \sum_{\gamma_{\chi}>\lambda_{\chi}>0} \frac{ \sqrt{m_{\gamma_{\chi}}m_{\lambda_{\chi}}}\Re(Z_{\gamma_{\chi}}) \Re(Z_{\lambda_{\chi}})}{ \sqrt{\frac 14+\gamma_{\chi}^2} \sqrt{\frac 14+\lambda_{\chi}^2}}\\
&+2 \sum_{\substack{\chi \in C(q) \\ \chi \text{ complex}}} \sum_{\gamma_{\chi}\neq 0} \frac{m_{\gamma_{\chi}}|Z_{\gamma_{\chi}}|^2}{ \frac 14+\gamma_{\chi}^2} + 4 \sum_{\substack{\chi \in C(q) \\ \chi \text{ complex}}} \sum_{\substack{\gamma_{\chi},\lambda_{\chi} \neq 0 \\ \gamma_{\chi}>\lambda_{\chi}}} \frac{\sqrt{m_{\gamma_{\chi}}m_{\lambda_{\chi}}} Z_{\gamma_{\chi}} \overline{Z_{\lambda_{\chi}}}}{ \sqrt{\frac 14+\gamma_{\chi}^2} \sqrt{\frac 14+\lambda_{\chi}^2}}.
\end{align*} 
At this point we run into the problem that the random variables in this expression are not all mutually independent, hence this sum of random variables does not satisfy the hypotheses of classical theorems on large deviations. We will use an alternative approach based on bounds on the moment-generating function (Laplace transform) of $H_q$, which we will then transfer into bounds on large deviations of $H_q$.

While the moment-generating functions of $\Re(Y_{\chi})$ and $\Im(Y_{\chi})$ can be computed explicitly in terms of Bessel functions (see Lemma \ref{lemma moment generating of Y_chi}), we were not able to find such a nice closed formula for $H_q$. 
%
%
We begin this section with an effective Stirling Formula.

\begin{lemma}[Stirling's Formula]
We have for $n\geq 2$ that
$$ 2.506...=\sqrt{2\pi}<\frac{n!}{\sqrt n (n/e)^n} < \sqrt{2\pi}e^{\frac 1{24}}= 2.613...$$
\label{lemma Stirling}
\end{lemma}
\begin{proof}
See 6.1.42 of \cite{Ab}.
\end{proof}

\begin{lemma}
\label{lemma moments of Y_chi}
Assume GRH and let $Y_{\chi}$ be the random variable defined in \eqref{equation definition of Y_chi}. We have for $q^*$ large enough and for $n\geq 1$ that 
$$ \E[|Y_{\chi}|^{2n}] \leq 5.7 n^{\frac 32} \left( \frac {4n\log q^*} {e-o(1)} \right)^n. $$
\end{lemma}

\begin{proof}

We will use the explicit formula for the moment-generating functions of $\Re(Y_{\chi})$ and $\Im(Y_{\chi})$ appearing in Lemma \ref{lemma moment generating of Y_chi}.
Note that $I_0(z)$ is an entire function of $z$,
and so the absolutely convergent product \eqref{equation moment generating of Y_chi} is also an entire function. The Taylor series of $I_0(z)$ gives the following immediate bound:
$$ |I_0(z)| \leq  \sum_{n=0}^{\infty} \frac 1{n!^2} \left( \frac {|z|}2\right)^{2n} \leq  \sum_{n=0}^{\infty} \frac 1{n!} \left( \frac {|z|^2}4\right)^{n} = e^{|z|^2/4}.  $$
For real characters $\chi$, this gives a bound on $\mathcal L_{Y_{\chi}}(z)$, since by \eqref{equation moment generating of Y_chi},
$$|\mathcal L_{Y_{\chi}}(z)| \leq  \exp\left(\sum_{\gamma_{\chi}>0} \frac{|z|^2m_{\gamma_{\chi}}}{\frac 14+\gamma_{\chi}^2} \right)\leq \exp\left(\frac{|z|^2}2 (\log q^*+O(\log\log q^*))\right) $$
for $q^*$ large enough, by Lemma \ref{lemma b(chi)}. In the last equation the sums over $\gamma_{\chi}$ are counted without multiplicity, and $m_{\gamma_{\chi}}$ denotes the multiplicity of $\rho_{\chi} = \frac 12+i\gamma_{\chi}$. We now use this to bound the moments of $Y_{\chi}$. Cauchy's formula for the derivatives reads
$$  \E[Y_{\chi}^{2n}]  = \frac{(2n)!}{2\pi i} \int_{|z|=C} \mathcal L_{Y_{\chi}}(z) \frac{dz}{z^{2n+1}},  $$
and so by our bound on $\mathcal L_{Y_{\chi}} (z)$,
$$ |\E[Y_{\chi}^{2n}]| \leq (2n)! \exp \left(\frac{C^2 \log q^*}{2-o(1)}\right) C^{-2n}. $$
Taking $n\geq 1$ and $C=(2n/\log q^*)^{\frac 12}$, we obtain
$$ \E[Y_{\chi}^{2n}] \leq (2n)! \left( \frac{2n(1+o(1))}{e\log q^*}\right)^{-n}, $$
which by applying Lemma \ref{lemma Stirling} gives the bound
$$ \E[Y_{\chi}^{2n}] \leq3.7 n^{\frac 12} \left( \frac {2n\log q^*} {e-o(1)} \right)^n. $$

If $\chi$ is complex, then we apply the above argument to the moment-generating function of $\Re(Y_{\chi})$ and $\Im(Y_{\chi})$ (see \eqref{equation moment generating of Y_chi complex}). Doing so, we obtain the following bound:
$$ \E[\Re(Y_{\chi})^{2n}],  \E[\Im(Y_{\chi})^{2n}]\leq 3.7 n^{\frac 12} \left( \frac {2n\log q^*} {e-o(1)} \right)^n. $$
We finish the proof by combining this with Lemma \ref{lemma Stirling} and the Cauchy-Schwartz inequality in the form $|\E[XY]| \leq \E[X^2]^{\frac 12}\E[Y^2]^{\frac 12}$ for real random variables $X, Y$:
\begin{align*}
&\E[|Y_{\chi}|^{2n}]= \E[( \Re(Y_{\chi})^2+\Im(Y_{\chi})^2)^{n}] = \sum_{k=0}^n \binom nk \E[\Re(Y_{\chi})^{2k}\Im(Y_{\chi})^{2(n-k)}] \\
&\leq \sum_{k=0}^n \binom nk \E[\Re(Y_{\chi})^{4k}]^{\frac 12} \E[\Im(Y_{\chi})^{4(n-k)}]^{\frac 12} \\
& \leq 5.7 \sum_{k=1}^{n-1}  \frac{(n/e)^n n^{\frac 12}}{ (k/e)^k k^{\frac 12} ((n-k)/e)^{n-k} (n-k)^{\frac 12}} k^{\frac 14} \left( \frac {4k\log q^*} {e-o(1)} \right)^k (n-k)^{\frac 14} \left( \frac {4(n-k)\log q^*} {e-o(1)} \right)^{n-k}\\
&\hspace{1cm} + 2\cdot 3.7^{\frac 12} (2n)^{\frac 14} \left( \frac {4n\log q^*} {e-o(1)} \right)^n \\
&\leq 5.7 \sum_{k=1}^n  n^{\frac 12}\left( \frac {4n\log q^*} {e-o(1)} \right)^{n} =5.7  n^{\frac 32}\left( \frac {4n\log q^*} {e-o(1)} \right)^{n} .
\end{align*}  

\end{proof}

\begin{remark}
An important fact used in the last proof is that $\mathcal L_{\Re(Y_{\chi})}(z)$ and $\mathcal L_{\Im(Y_{\chi})}(z)$ are entire functions, which we integrated on the circle $|z|=(2n/\log q^*)^{\frac 12}$, whose radius tends to infinity with $n$. This would not have worked with the cumulant-generating function $\log \E[e^{iz\Re(Y_{\chi})}]$, which has poles by \eqref{equation moment generating of Y_chi} since $I_0(z)$ has infinitely many zeros on the imaginary axis.  
\end{remark}

Now that we have bounded the moments of $|Y_{\chi}|^2$, we will turn this information into a bound on $\mathcal L_{H_q}(z)$, the moment-generating function of $H_q$. Instead of studying the moments $H_q$ itself, we will study its centered moments, by defining
$$ \tilde{H}_q := H_q-\E[H_q] = \sum_{\substack{\chi\bmod q \\ \chi \neq \chi_0}} W_{\chi}, $$
where (see \eqref{equation definition of H})
\begin{equation}
 W_{\chi} := |Y_{\chi}|^2-\E[|Y_{\chi}|^2]. 
 \label{equation definition of W_chi}
\end{equation}
\begin{lemma}
\label{lemma bound moment generating of H}
Assume GRH. For $q$ large enough, the moment-generating function of $\tilde H_q$ satisfies, in the range $|t|<(40\log q)^{-1}$, 
\begin{equation}
\mathcal L_{\tilde H_q}(t) \leq (1+184t^2(\log q)^2)^{\phi(q)}.
\label{equation borne sur L H_q}
\end{equation} 
\end{lemma}

\begin{proof}

Let $0\leq t < (40\log q)^{-1} $. Using the identity $\mathcal L_{X+c}(z) = e^{zc} \mathcal L_X(z)$, we have for $q$ large enough and $W_{\chi}$ defined as in \eqref{equation definition of W_chi} that
\begin{align}
\begin{split}
\mathcal L_{W_{\chi}}(t) &= e^{-t\E[|Y_{\chi}|^2]} \mathcal L_{|Y_{\chi}|^2}(t) \\
& = e^{-t\E[|Y_{\chi}|^2]}\sum_{n=0}^{\infty}\E[|Y_{\chi}|^{2n}]  \frac{t^n}{n!} \\
&  \leq (1-t\E[|Y_{\chi}|^2] +0.561 t^2 \E[|Y_{\chi}|^2]^2) \sum_{n=0}^{\infty}\E[|Y_{\chi}|^{2n}]  \frac{t^n}{n!}  \\
&= 1+\sum_{n=2}^{\infty} t^n \left( \frac{ \E[|Y_{\chi}|^{2n}]}{n!} - \frac{\E[|Y_{\chi}|^2]\E[|Y_{\chi}|^{2n-2}]}{(n-1)!} +0.561  \frac{\E[|Y_{\chi}|^2]^2\E[|Y_{\chi}|^{2n-4}]}{(n-2)!}  \right) .
\end{split}
\label{equation bound L W chi}
\end{align}  

Now, by Lemma \ref{lemma moments of Y_chi}, we have for $n\geq 1$ that
$$ \E[|Y_{\chi}|^{2n}] \leq 5.7 n^{\frac 32} \left( \frac{4n \log q} {e-o(1)}\right)^{n  }. $$
Hence, applying Lemma \ref{lemma Stirling},  we obtain
\begin{align*}
\mathcal L_{W_{\chi}}(t) &\leq 1+ 4.65\sum_{n=2}^{\infty} n^{\frac 32} \left( \frac{4n \log q} {e-o(1)}\right)^{n}t^n \left( \frac en\right)^n n^{-\frac 12} \\
& =  1+ 4.65 \sum_{n=2}^{\infty}  n((4+o(1))t\log q)^n \\
&\leq 1+ 184 t^2(\log q)^2, 
\end{align*}
as long as $0\leq t<(40\log q)^{-1}$. Since the $W_{\chi}$ are all mutually independent, the proof follows by multiplicativity:
$$ \mathcal L_{\tilde H_q}(t)  = \prod_{\substack{\chi\bmod q \\ \chi \neq \chi_0}} \mathcal L_{W_{\chi}}(t).  $$
A similar argument works in the range $-(40\log q)^{-1} < t <0$.
%
%

%
%

\end{proof}
%
%
%

%
%
%
%
%

\begin{proof}[Proof of Theorem \ref{theorem large deviations}, upper bound]

%

The first estimate for $\E[H_q]$ appearing in Theorem \ref{theorem first two moments} implies that under GRH, for $q$ large enough and for $(\log\log q)^2/\log q \leq \epsilon \leq 1$ (recall $\tilde H_q = H_q-\E[H_q]$) we have
$$ \P[H_q>(1+\epsilon) \phi(q) \log q] \leq \P[\tilde H_q>0.99\epsilon \phi(q) \log q]. $$

As is customary (this is Chernoff's inequality), we relate the large deviations of $\tilde H_q$ to its moment-generating function using Markov's inequality:
$$ \P[\tilde H_q>V] = \P[e^{t\tilde H_q}>e^{tV}] \leq e^{-tV} \E[e^{t\tilde H_q}]. $$
Taking $V=0.99\epsilon \phi(q) \log q$ we obtain that
$$ \P[H_q>(1+\epsilon) \phi(q) \log q] \leq \exp(-0.99t \epsilon \phi(q) \log q) \mathcal L_{\tilde H_q}(t), $$
which from Lemma \ref{lemma bound moment generating of H} is, for $t=\epsilon/(370\log q)$ (we use that $\epsilon\leq 9$), 
\begin{align*}
&\leq \exp(-0.99\epsilon^2 \phi(q)/370) (1+184\epsilon^2/370^2)^{\phi(q)} \\ &= \exp[\phi(q) (-0.99\epsilon^{2}/370+\log (1+184\epsilon^2/370^2))] \\ &\leq \exp(-\epsilon^2\phi(q)/751),
\end{align*} 
since for $x\geq 0$, $\log(1+x)\leq x$.
We have therefore established the bound 
$$  \P[H_q>(1+\epsilon) \phi(q) \log q] \leq \exp(-\epsilon^2\phi(q)/751). $$
We conclude the proof by writing
$$ \P[\tilde H_q < -V] = \P[-\tilde H_q > V]= \P[e^{-t\tilde H_q} > e^{tV} ] \leq e^{-tV} \E[e^{-t \tilde H_q}], $$
and by applying the same reasoning to $\mathcal L_{\tilde H_q} (-t)$, from which we deduce that
$$  \P[H_q<(1-\epsilon) \phi(q) \log q] \leq \exp(-\epsilon^2\phi(q)/751). $$
\end{proof}

For the lower bound we will use the following inequality.
\begin{lemma}[Paley-Zygmund Inequality]
If $X\geq 0$ is a random variable having a second moment, then for any $0<a<1$ we have
$$ \P[X\geq a\E[X]] \geq (1-a)^2 \E[X]^2/\E[X^2].  $$
\label{lemma lower bound}
\end{lemma}
\begin{proof}
Let $I\subset \mathbb R$ be an interval and define the random variable
$\mathbf 1_{X \in I}$ as follows: 
$$\mathbf 1_{X \in I} := \begin{cases}
1 &\text{ if } X \in I\\
0 & \text{ otherwise,}
\end{cases} $$
so that $\E[\mathbf 1_{X \in I}]=\P[X\in I]$. Using this notation we have for any $U > 0$ that
 \begin{align*}
 \E[X] =\E[X\mathbf 1_{X < U} ] +  \E[X\mathbf 1_{X \geq U} ] \leq U + \E[X^2]^{\frac 12} \P[X\geq U]^{\frac 12}
 \end{align*}
by the Cauchy-Schwartz inequality and the fact that $\mathbf 1_{X \in I}^2 =\mathbf 1_{X \in I}$. The proof follows by taking $U=a\E[X]$.
 
\end{proof}

\begin{proof}[Proof of Theorem \ref{theorem large deviations}, lower bound]
By Lemma \ref{lemma lower bound}, we have that if $t$ and $V$ are such that
$ e^{tV} \leq \E[e^{t\tilde H_q}]/2 $, then
$$ \P[\tilde H_q \geq V] = \P[e^{t\tilde H_q} \geq e^{tV}] \geq \frac 14 \frac{\E[e^{t \tilde H_q}]^2}{\E[e^{2t\tilde H_q}]}. $$
Taking $V=(1+\epsilon) \phi(q) \log q-\E[H_q]$, we need to select $t$ for which $ e^{tV} \leq \E[e^{t\tilde H_q}]/2 $. We start with $\chi$ real. By lemmas \ref{lemma Stirling} and \ref{lemma moments of Y_chi}, we have that in the range $0<t<(100\log q^*)^{-1}$, 
\begin{align*}
\E[e^{t |Y_{\chi}|^2}] &= 1+ t \E[|Y_{\chi}|^2] + \frac{t^2}2 \E[|Y_{\chi}|^4] + \overline{O}\left( 5.7 \sum_{n\geq 3} n^{\frac 32} \left(\frac{4n\log q^*}{e-o(1)}\right)^n\frac{|t|^n}{n!}  \right) \\
&= 1+ t \E[|Y_{\chi}|^2] + \frac{t^2}2 \E[|Y_{\chi}|^4] + \overline{O}\left( 2.28\frac {3(4|t|\log q^*)^3}{(1-4|t|\log q^*)^2} \right) \\
&= 1+ t \E[|Y_{\chi}|^2] +\frac{t^2}2 \E[|Y_{\chi}|^4] +  \overline{O}\left(475   (|t|\log q^*)^3 \right), 
\end{align*} 
where $\overline O$ means that the implied constant is one. Therefore, in this range of $t$ and for $q$ large enough, one shows using the estimates $\E[|Y_{\chi}|^2]\sim \log q^*$ and $\E[|Y_{\chi}|^4] \sim 3(\log q^*)^2$ obtained in the proof of Theorem \ref{theorem first two moments} (see \eqref{equation 1 lemma first two moments} and \eqref{equation fourth moment of Y_chi}) that for $0<t <(2100\log q^*)^{-1}$,
$$ \E[e^{t W_{\chi}}] = e^{-t\E[|Y_{\chi}|^2]}  \E[e^{t |Y_{\chi}|^2}] = 1+\frac{t^2}2(\E[|Y_{\chi}|^4] - \E[|Y_{\chi}|^2]^2)+ \overline O\left( 481 (|t|\log q^*)^3 \right). $$
Hence,
for $q^*$ large enough and for $0<t <(2100\log q^*)^{-1}$,
$$\E[e^{t W_{\chi}}]  \geq  e^{ 0.77 t^2 (\log q^*)^2}. $$
For complex $\chi$, we obtain a similar estimate using \eqref{equation 1 lemma first two moments complex} and \eqref{equation 4th moment complex chi}, with the constant $0.77$ replaced with $1.54$. 
  This shows that
\begin{align}
\E[e^{t\tilde H_q}]  = \prod_{\chi \in C(q) } \E[e^{t W_{\chi}}]   \geq \exp\bigg( 0.77 t^2 \sum_{\substack{\chi\bmod q \\ \chi \neq \chi_0}}(\log q^*)^2\bigg) &\geq e^{0.75t^2\phi(q) (\log q)^2} \label{equation lower bound on moment generating} \\ & \geq 2 e^{tV}, \notag 
\end{align} 
for $q$ large enough, $(\log\log q)^2/\log q < \epsilon < 3000^{-1}$
and $ 1.4 \epsilon / \log q \leq t <(2100 \log q^*)^{-1}$, since Theorem \ref{theorem first two moments} shows that in this range,
$$V=(1+\epsilon) \phi(q) \log q-\E[H_q] \sim \epsilon \phi(q)\log q .$$ We conclude by Lemma \ref{lemma bound moment generating of H} and \eqref{equation lower bound on moment generating} that
$$ \P[\tilde H_q \geq V] \geq \frac 14 \frac{\E[e^{t \tilde H_q}]^2}{\E[e^{2t\tilde H_q}]} \geq \frac 14 e^{1.5t^2 \phi(q)(\log q)^2}(1+184(2t\log q)^2)^{-\phi(q)} . $$
Taking $t= 1.4\epsilon/\log q $ gives the result.

\end{proof}

\section{Concluding Remarks}

Going from Theorem \ref{theorem large deviations} to Conjecture \ref{conjecture Hooley in large range} is not direct. Indeed, if we are studying the quantity $V(x;q)$ for $q$ and $x$ in a given range such as $(\log\log x)^{1+\delta} <q \leq x^{o(1)}$, then it is not clear that the limiting logarithmic distribution of $V(x;q)$ coincides with that of $H_q$. 
Indeed one would need to show that in the range $(\log Y)^{1+\delta} < q \leq e^{o(Y)}$, we have for every fixed $m\geq 1$ that
\begin{equation}
 \frac 1Y\int_0^Y (\phi(q)e^{-y}V(e^y;q))^m dy \sim \E[H_q^m].
 \label{equation mth moment conjecture}
\end{equation}  
This last integral is similar to a $2m$-correlation sum of low-lying zeros of Dirichlet $L$-functions. Indeed, expanding the $m$-th power we obtain from \eqref{equation expression for V(x;q)} under GRH that in the range $(\log Y)^{1+\delta} < q \leq e^{o(Y)}$,
\begin{align}
\begin{split}
&\int_0^Y (\phi(q) e^{-y}V(e^y;q)-o(1))^m dy= \int_0^Y \left(\sum_{\substack{\chi\bmod q \\ \chi \neq \chi_0}}  \sum_{\gamma_{\chi},\gamma_{\chi}'}\frac{e^{i(\gamma_{\chi}-\gamma_{\chi}') y } }{(\frac 12+i\gamma_{\chi})(\frac 12-i\gamma_{\chi}')} \right)^m dy \\
&=  \sum_{\chi_1,...,\chi_m \neq \chi_0} \sum_{\gamma_{\chi_1},\gamma_{\chi_1}',...,\gamma_{\chi_m},\gamma_{\chi_m}'} \frac{ \int_0^Y e^{i(\gamma_{\chi_1}+...+\gamma_{\chi_m}-\gamma_{\chi_1}'-...-\gamma'_{\chi_m}) y } dy  }{(\frac 12+i\gamma_{\chi_1})(\frac 12-i\gamma_{\chi_1}') \cdots (\frac 12+i\gamma_{\chi_m})(\frac 12-i\gamma_{\chi_m}') } \\
&=  Y\sum_{\chi_1,...,\chi_m \neq \chi_0} \sum_{\substack{\gamma_{\chi_1},\gamma_{\chi_1}',...,\gamma_{\chi_m},\gamma_{\chi_m}' \\ \gamma_{\chi_1}-\gamma_{\chi_1}'+...,+\gamma_{\chi_m}-\gamma_{\chi_m}'=0 }} \frac 1{(\frac 12+i\gamma_{\chi_1})(\frac 12-i\gamma_{\chi_1}') \cdots (\frac 12+i\gamma_{\chi_m})(\frac 12-i\gamma_{\chi_m}') } 
\\& 
+  \sum_{\substack{\chi_1,...,\chi_m \neq \chi_0 \\ \gamma_{\chi_1},\gamma_{\chi_1}',...,\gamma_{\chi_m},\gamma_{\chi_m}' \\ \gamma_{\chi_1}-\gamma_{\chi_1}'+...+\gamma_{\chi_m}-\gamma_{\chi_m}' \neq 0}} 
\frac { (e^{i(\gamma_{\chi_1}+...+\gamma_{\chi_m}-\gamma_{\chi_1}'-...-\gamma'_{\chi_m}) Y }-1)(\gamma_{\chi_1}+...+\gamma_{\chi_m}-\gamma_{\chi_1}'-...-\gamma'_{\chi_m})^{-1} }{(\frac 12+i\gamma_{\chi_1})(\frac 12-i\gamma_{\chi_1}') \cdots (\frac 12+i\gamma_{\chi_m})(\frac 12-i\gamma_{\chi_m}') }  .
\end{split}
\label{equation mth moment}
\end{align} 
If the last sum was running over the zeros of a single $L$-function, then we would run into the problem that if two zeros $\gamma,\gamma'$ are extremely close to each other, then $e^{i(\gamma-\gamma')Y}(\gamma-\gamma')$ is very close to $Y$, giving a significant contribution to \eqref{equation mth moment}. However in the present situation we are taking an average over all Dirichlet $L$-functions modulo $q$, and hence the number of pairs of such zeros will be negligible compared to the size of the family we average over, under assumptions on statistics on zeros of Dirichlet $L$-functions.  


We now show how a conjecture on the pair correlation of low-lying zeros of Dirichlet $L$-functions implies that \eqref{equation mth moment conjecture} holds for $m=1$ in the range $(\log\log x)^{1+\delta} \leq q \leq x^{o(1)}$. Using Schlage-Puchta's method \cite{Pu}, one shows that the last term in \eqref{equation mth moment} is an error term for fixed values of $q$, and from this we can conclude under GRH and LI that
$$ \E[H_q^m] = \sum_{\chi_1,...,\chi_m \neq \chi_0} \sum_{\substack{\gamma_{\chi_1},\gamma_{\chi_1}',...,\gamma_{\chi_m},\gamma_{\chi_m}' \\ \gamma_{\chi_1}-\gamma_{\chi_1}'+...,+\gamma_{\chi_m}-\gamma_{\chi_m}'=0 }} \frac 1{(\frac 12+i\gamma_{\chi_1})(\frac 12-i\gamma_{\chi_1}') \cdots (\frac 12+i\gamma_{\chi_m})(\frac 12-i\gamma_{\chi_m}') }  .$$
(This actually follows from Theorem \ref{theorem large deviations}, with the same method as in Lemma 2.5 of \cite{Fi}.)
Thus \eqref{equation mth moment conjecture} reduces to the statement that the last term in \eqref{equation mth moment} is an error term for every fixed $m$ and for values of $q$ not necessarily fixed.
Taking $m=1$, the last term in \eqref{equation mth moment} is 
$$ T(Y;q) := \sum_{\substack{\chi\bmod q \\ \chi \neq \chi_0}} \sum_{\gamma_{\chi}\neq\gamma_{\chi}'} \frac{(e^{i(\gamma_{\chi}-\gamma_{\chi}')Y}-1)(\gamma_{\chi}-\gamma_{\chi}')^{-1}}{(\frac 12+i\gamma_{\chi})(\frac 12-i\gamma_{\chi}') } \ll \sum_{\chi \bmod q} \sum_{\gamma_{\chi}\neq \gamma_{\chi}'} \frac{\min(Y,|\gamma_{\chi}-\gamma_{\chi}'|^{-1})}{(1+|\gamma_{\chi}|)(1+|\gamma_{\chi}'|) }  , $$
and we would like to show that in the range $(\log Y)^{1+\delta} \leq q \leq e^{o(Y)}$ we have $T(Y;q)  = o(Y\E[H_q]),$ that is $T(Y;q) =o(Y \phi(q) \log q)$. Arguing as in Lemma 2.6 of \cite{Fi} (this is Schlage-Puchta's technique \cite{Pu}), we have introducing a parameter $U\geq 1$ that
\begin{multline}
 T(Y;q) \ll  \sum_{\chi \bmod q} \sum_{\substack{\gamma_{\chi},\gamma_{\chi}' \\ |\gamma_{\chi}-\gamma_{\chi}'| \geq 1}} \frac{|\gamma_{\chi}-\gamma_{\chi}'|^{-1}}{(1+|\gamma_{\chi}|)(1+|\gamma_{\chi}'|) }  +  \sum_{\chi \bmod q} \sum_{\substack{\gamma_{\chi},\gamma_{\chi}' \\ 0<|\gamma_{\chi}-\gamma_{\chi}'| \leq 1 \\ \gamma_{\chi},\gamma_{\chi}' > U }} \frac{\min(Y,|\gamma_{\chi}-\gamma_{\chi}'|^{-1})}{(1+|\gamma_{\chi}|)(1+|\gamma_{\chi}'|) }  \\ + \sum_{\chi \bmod q} \sum_{\substack{\gamma_{\chi},\gamma_{\chi}' \\ 0<|\gamma_{\chi}-\gamma_{\chi}'| \leq 1 \\ \gamma_{\chi},\gamma_{\chi}' \leq U }} \frac{\min(Y,|\gamma_{\chi}-\gamma_{\chi}'|^{-1})}{(1+|\gamma_{\chi}|)(1+|\gamma_{\chi}'|) }  = I+II+III.
 \end{multline}
We compare the first term with an integral:
$$ I \ll \phi(q)\iint_{\substack{ x,y \geq 0 \\ |x-y|\geq 1}} \frac{\log(qx) \log(qy) dxdy}{ |x-y| (x+1) (y+1)} \ll \phi(q) (\log q)^2, $$
 which is $o(Y\phi(q) \log q)$ as soon as $\log q = o(Y)$ (this holds in our range of $q$). As for the second term, we have by the Riemann-von Mangoldt formula that
 $$ II \ll \sum_{\chi \bmod q} \sum_{\gamma_{\chi} \geq U} \frac {Y \log \gamma_{\chi} }{(1+\gamma_{\chi})^2}  \ll Y\phi(q) \frac{(\log (qU))^2}U. $$
The third term is the hardest, and requires to make the following conjecture on the pair correlation of zeros of Dirichlet $L$-functions:

\begin{conjecture}
Fix $C>0$. There exists a bounded function $W(t) \geq 0$ such that in the range $ C^{-1} \leq  S \leq \log q$, $1\leq U \leq C \log Q $ we have
\begin{equation}
\sum_{\chi \bmod q} \#\{ 0\leq \gamma_{\chi}, \gamma_{\chi}' \leq U : 0<|\gamma_{\chi}-\gamma_{\chi}' | \leq S/\log q \} \ll_C \phi(q) U\log (qU)  \int_{0}^S W(t) dt.
\label{conjecture pair correlation}
\end{equation} 
\end{conjecture}
Assuming this conjecture, we have by fixing $\epsilon>0$ and using summation by parts that
\begin{align*}
III &\leq  \sum_{\chi \bmod q} \sum_{\substack{\gamma_{\chi},\gamma_{\chi}' \\ 0<|\gamma_{\chi}-\gamma_{\chi}'| \leq \frac{\epsilon}{\log q} \\ \gamma_{\chi},\gamma_{\chi}' \leq U }} \frac{Y}{(1+|\gamma_{\chi}|)(1+|\gamma_{\chi}'|) } +  \sum_{\chi \bmod q} \sum_{\substack{\gamma_{\chi},\gamma_{\chi}' \\ \frac{\epsilon}{\log q}  \leq  |\gamma_{\chi}-\gamma_{\chi}'| \leq 1\\ \gamma_{\chi},\gamma_{\chi}' \leq U }} \frac{|\gamma_{\chi}-\gamma_{\chi}'|^{-1}}{(1+|\gamma_{\chi}|)(1+|\gamma_{\chi}'|) } \\
& \ll Y\phi(q) \epsilon \log q +  \phi(q) \frac{\log q}{\epsilon} \log q . 
\end{align*}  

Collecting all these terms we obtain that
$$ I+II+III \ll \phi(q) (\log q)^2 +Y\phi(q) \frac{(\log (qU))^2}U+Y\phi(q) \epsilon \log q +  \phi(q) \frac{\log q}{\epsilon} \log q ,  $$
which by taking $U=(\log q)/\epsilon $ is 
$$ \ll \phi(q) (\log q)^2  +  \epsilon Y\phi(q) \log q + \phi(q) \frac{(\log q)^2}{\epsilon}, $$
a quantity which is $= o(Y\phi(q) \log q)$ when $q$ is in the range $(\log Y)^{1+\delta} \leq q \leq e^{o(Y)} $. This justifies why \eqref{equation mth moment conjecture} should hold in this range.

One can justify \eqref{equation mth moment conjecture} for all $m\geq 1$ with a similar argument, under an assumption on the statistics of zeros of Dirichlet $L$-functions. 

Note that in the range $(\log x)^{1+\delta}  \leq q \leq x^{o(1)}$, knowing the first two moments of $e^{-y/2}V(e^y;q)$ is sufficient for justifying \eqref{equation hooley conjecture}. Indeed, Chebyshev's inequality shows that
$$\P[|H_q - \E[H_q]| >\epsilon\phi(q) \log q] \ll \frac 1{\epsilon^2 \phi(q)}, $$
and the following argument gives the desired result in this range of $q$.

We now show how to support Conjecture \ref{conjecture Hooley in large range}, assuming that \eqref{equation mth moment conjecture} holds in the range $(\log Y)^{1+\delta} \leq q \leq e^{o(Y)}$ (the reason why we chose this upper bound is that in the range $(\log x)^{1+\delta}  \leq q \leq x^{o(1)}$ we only need \eqref{equation mth moment conjecture} to hold for $m=1,2$). Theorem \ref{theorem large deviations} shows that under GRH and LI, 
$$ meas\{ y\leq Y : \phi(q) e^{-y}V(e^y;q) \in (\alpha,\beta) \} \sim Y \P[H_q\in (\alpha,\beta)], $$
and so in the range $(\log Y)^{1+\delta} \leq q \leq e^{o(Y)}$, Theorem \ref{theorem large deviations} gives
\begin{align*} meas\{ y\leq Y : e^{-y}V(e^y;q) &\notin ( (1-\epsilon)\log q,(1+\epsilon)\log q ) \} \\&\sim Y \P[\tilde H_q\notin (-\epsilon \phi(q)\log q,\epsilon\phi(q) \log q)] \\
& \leq 2Y \exp \left(-c_2\epsilon^2 \phi(q) \right).
\end{align*}

However, since $e^{-y}V(e^y;q)$ can be understood by looking at the equidistribution of the vector $(e^{i\gamma_1y},...,e^{i\gamma_ky})\in \mathbb T^k$, we expect that the smallest value of $y$ for which $|e^{-y}V(e^y;q)-\log q|\neq o(\log q)$ is about $y\approx \exp \left(  c \phi(q) \right)$. That is to say, for $q\geq (\log y)^{1+\delta}$ we have $e^{-y}V(e^y;q)\sim \log q$, which is equivalent to Conjecture \ref{conjecture Hooley in large range}.

\begin{remark}
Theorem \ref{theorem large deviations} and Remark \ref{remark extended range epsilon} even suggest the following estimate, for \\$(\log\log x)^{1+\delta}\leq q \leq x^{o(1)}$:
\begin{equation}
\label{equation very precise conjecture}
V(x;q) 
= x \mathcal L(q) \left( 1+O\left( \Psi(x) \sqrt{\frac{\log\log x}{\phi(q)}} \right)\right),
\end{equation} 
where $\Psi(x)$ is any function tending to infinity with $x$ and $$\mathcal L(q):= \log q -\gamma -\log (2\pi) -\sum_{p\mid q} \frac{\log p}{p-1} .  $$
\end{remark}

\begin{remark}

It would be interesting to investigate the large deviations of $H_q$ in Theorem \ref{theorem large deviations} for larger values of $\epsilon$. Indeed we believe that a transition happens near $\epsilon\asymp 1$, and this could give information about $V(x;q)$ in the range $q\leq \log\log x$. For example one could make a prediction on the best possible bound for $V(x;q)$ in this range. If $q$ is fixed, then one can show using \eqref{equation H n terms of characters} that the limiting distribution of $e^{-y}V(e^y;q)$ has double-exponentially decaying tails (this follows from Montgomery's work \cite{Mo}), resulting in the prediction
\begin{equation}
\label{equation montgomery bound}
 V(x;q) \ll x (\log\log\log x)^4.
\end{equation} 
Again this is for fixed values of $q$, and shows that a transition happens in the range $1\leq q \leq (\log\log x)^{1+\delta}$, in transferring from \eqref{equation very precise conjecture} to \eqref{equation montgomery bound}.

\end{remark}

\section*{Acknowledgements}
I would like to thank John Friedlander, Jeffrey C. Lagarias, Steven J. Miller, Hugh L. Montgomery and Maksym Radziwill for their useful comments. I also thank Greg Martin for his help with Lemma \ref{lemma FiMa}. I thank V\'ictor P\'erez Abreu for introducing me to Berg's paper, for fruitful conversations and for inviting me to the CIMAT in Guanajuato, M\'exico. This work was accomplished partly at the Institute for Advanced Study and at the University of Michigan, and was supported by an NSERC Postdoctoral Fellowship as well as NSF grant DMS-0635607.

\appendix

\section{Some comments on Montgomery's Conjecture}

Montgomery's Conjecture \cite{M} states that if we fix $\epsilon>0$, then
$$\psi(x;q,a) = \frac{x}{\phi(q)} +O_{\epsilon}\left(\frac{x^{\frac 12+\epsilon}}{q^{\frac 12}} \right), $$
uniformly for all coprime integers $a,q \leq x^{\frac 12}$.  
The Generalized Riemann Hypothesis implies that for $\chi \neq \chi_0$, 
$$ \psi(x,\chi) := \sum_{n\leq x} \Lambda(n) \chi(n) \ll x^{\frac 12} (\log x)^2, $$
and thus using orthogonality relations we obtain the following standard estimate:
$$ \psi(x;q,a) - \frac{\psi(x,\chi_0)}{\phi(q)} =\frac 1{\phi(q)} \sum_{\substack{\chi\bmod q \\ \chi \neq \chi_0}} \overline{\chi}(a) \psi(x,\chi) \ll x^{\frac 12} (\log x)^2.$$
Montgomery's Conjecture is motivated by the fact that we did not exploit any cancellations in the last estimate. Indeed, since under GRH and LI the quantity $\overline{\chi}(a)e^{-y/2}\psi(e^y,\chi)$ is distributed like a random variable of mean zero and variance asymptotically $\log q^*$, we have that the limiting distribution of $e^{-y/2}\sum_{\substack{\chi\bmod q \\ \chi \neq \chi_0}} \overline{\chi}(a) \psi(e^y,\chi)$ has variance asymptotically $\phi(q) \log q$. This means that this last sum is normally of order $(\phi(q) \log q)^{\frac 12+\epsilon}$, which in turn gives that $\psi(x;q,a) - \frac{\psi(x,\chi_0)}{\phi(q)}$ is normally of order $x^{\frac 12} / \phi(q)^{\frac 12-\epsilon}$, that is Montgomery's Conjecture should hold.

As for the quantity $\psi(x,\chi)$, we do not expect any improvement under GRH of the form 
$$\psi(x,\chi) \ll \frac{x^{\frac 12 +\epsilon}}{q^{\theta}},$$
for any fixed $\theta >0$. Indeed this last bound can readily be disproved under GRH using either \eqref{equation conditional lower bound on V(x;q)} or \eqref{equation Goldston Vaughan} (in the range $x/(\log x)^A \leq q\leq x$ it can even be disproved unconditionally using \eqref{equation montgomery theorem}). Moreover, we believe that it is interesting to see what Montgomery's Conjecture implies on this quantity, using the orthogonality relations. Indeed under this conjecture we have for $\chi \neq \chi_0$ that
$$ \psi(x,\chi) = \sum_{\substack{a \bmod q \\ (a,q)=1}} \chi(a) \left(\psi(x;q,a) - \frac{\psi(x,\chi_0)}{\phi(q)} \right) \ll \phi(q) \frac{x^{\frac 12+\epsilon}}{q^{\frac 12}} \ll q^{\frac 12} x^{\frac 12+\epsilon},  $$
which is worse that GRH. Again, the discrepancy between this and the 'true bound' $\psi(x,\chi) \ll x^{\frac 12} (\log x)^2$ comes from the fact that square-root cancellation occurs in the last sum.

\end{document}